\documentclass{amsart}
\usepackage{amssymb}
\usepackage{hyperref}
\usepackage[utf8x]{inputenc}
\DeclareMathOperator{\dist}{dist}
\DeclareMathOperator{\supp}{supp}
\DeclareMathOperator{\tail}{tail}
\newtheorem{theorem}{Theorem}[section]

\newtheorem{proposition}[theorem]{Proposition}
\newtheorem{corollary}[theorem]{Corollary}
\theoremstyle{remark}\newtheorem{remark}[theorem]{Remark}

\DeclareFontFamily{U}{mathx}{\hyphenchar\font45}
\DeclareFontShape{U}{mathx}{m}{n}{
      <5> <6> <7> <8> <9> <10>
      <10.95> <12> <14.4> <17.28> <20.74> <24.88>
      mathx10
      }{}
\DeclareSymbolFont{mathx}{U}{mathx}{m}{n}
\DeclareFontSubstitution{U}{mathx}{m}{n}
\DeclareMathAccent{\widecheck}{0}{mathx}{"71}
\DeclareMathAccent{\wideparen}{0}{mathx}{"75}

\title{Refined Strichartz inequalities for the wave equation}
\author[Terence L.~J.~Harris]{Terence L.~J.~Harris}
\address{Department of Mathematics, University of Illinois, Urbana, IL 61801, U.S.A.}
\email{terence2@illinois.edu}
\thanks{This material is based upon work partially supported by the National Science Foundation under Grant No. DMS-1501041. I would like to thank Burak Erdoğan for suggesting the problem and for advice on this topic, and for financial support.
}
\numberwithin{equation}{section}
\begin{document}

\begin{abstract} Some analogues of the Schrödinger refined Strichartz inequalities (Du, Guth, Li and Zhang) are obtained for the wave equation. These are used to improve the best known $L^2$ fractal Strichartz inequalities for the wave equation in dimensions $d \geq 4$. 
 \end{abstract}

\maketitle
\section{Introduction} 
A special case of the Strichartz inequality for the wave equation asserts that
\begin{equation} \label{classicalstrichartz} \lVert u\rVert _{L^q(\mathbb{R}^{d+1} ) } \lesssim \left\lVert u\left( \cdot, 0\right)\right\rVert _{\dot{H}^s\left(\mathbb{R}^d \right)} + \left\lVert  u_t\left(\cdot, 0 \right) \right\rVert _{\dot{H}^{s-1}\left( \mathbb{R}^d \right)} \quad \text{ for } \quad \Delta u - u_{tt} = 0, \end{equation}
where $d \geq 2$, $s=1/2$ and $q= \frac{2(d+1)}{d-1}$ \cite{strichartz,keel}. The purpose of this work is to give a refined version of this inequality in certain cases, analogous to the refined Strichartz inequalities for the Schrödinger equation in \cite{li,du,du2}. The refined inequalities for the Schrödinger equation were used in \cite{li,zhang} to prove sharp results on the almost everywhere pointwise convergence of Schrödinger solutions to the initial data. They were also used in \cite{du,zhang} to improve the known results on Falconer's distance set problem in dimensions 3 and higher. This was done using Mattila's approach via the $L^2$ spherical averages of Fourier transforms of measures \cite{mattila}. Here the refined Strichartz inequality for the wave equation will be used to improve results on the averages over the cone, rather than the sphere. It is known \cite{wolff,erdogan2,cho} that the optimal decay of these conical averages is related to best possible $s$ in an inequality of the form \eqref{classicalstrichartz}, with the $L^q$ norm on the left replaced by an $L^2(\mu)$ norm for a fractal measure $\mu$. 

Fix $n \geq 1$. The Fourier transform of a compactly supported Borel measure $\mu$ on $\mathbb{R}^n$ is given by 
\[ \widehat{\mu}(\xi) := \int e^{-2\pi i \langle \xi, x \rangle } \, d\mu(x). \]
The measure $\mu$ is called $\alpha$-dimensional if 
\[ c_{\alpha}(\mu) := \sup_{\substack{x \in \mathbb{R}^n \\ r>0 }} \frac{ \mu\left(B\left(x,r\right)\right)}{r^{\alpha}} < \infty. \]
Given a smooth compact surface $\Gamma \subseteq \mathbb{R}^n$, let $\beta(\alpha,\Gamma)$ be the supremum over all $\beta \geq 0$ which satisfy 
\begin{equation} \label{coneaverages} \int \left\lvert  \widehat{\mu}(R \xi) \right\rvert^2 \, d\sigma_{\Gamma}(\xi) \lesssim \lVert \mu\rVert c_{\alpha}(\mu) R^{-\beta} \quad \text{ for all } R >0, \end{equation}
uniformly over all Borel measures $\mu$ with support in the unit ball of $\mathbb{R}^{n}$. Here $\lVert \mu\rVert $ is the total variation norm, the implicit constant is allowed to depend on $\beta$, and $\sigma_{\Gamma}$ is the surface measure on $\Gamma$. The distance set problem is related to the case where $\Gamma = S^{n-1}$ \cite{mattila,liu}, see also \cite{luca} for the connection between $\beta\big(\alpha, S^{n-1}\big)$ and the pointwise convergence of solutions to the wave and Schrödinger equations. Henceforth $\Gamma$ will denote the truncated cone
\[ \Gamma := \left\{ \left(\xi, \lvert \xi\rvert\right) \in \mathbb{R}^{d+1} : 1 \leq \lvert \xi\rvert \leq 2 \right\}, \quad n =d+1, \quad d \geq 1, \]
and since the constant factor in \eqref{coneaverages} is not important, it may be assumed that $\sigma_{\Gamma}$ is actually the pushforward of the $d$-dimensional Lebesgue measure under the map $\xi \mapsto \left( \xi, \left\lvert \xi\right\rvert \right)$. 

For $d \geq 3$, the best known lower bounds are
\begin{equation}  \label{cho} \beta\big(\alpha,\Gamma^d\big) \geq \left\lbrace
\begin{array}{lll} 
\alpha &\text{ if } 0 < \alpha \leq \frac{d-1}{2} & \text{(Mattila \cite{mattila}, =Rogers \cite{rogers})}\\
  \frac{d}{4} + \frac{\alpha}{2} - \frac{1}{4}  &\text{ if } \frac{d-1}{2} < \alpha \leq \frac{d+3}{2} & \text{(Cho, Ham and Lee \cite{cho})}\\
\alpha-1 &\text{ if } \frac{d+3}{2} < \alpha \leq d+1 & \text{(Sjölin \cite{sjolin}),} \end{array} \right. \end{equation}
and the best known upper bounds are
\begin{equation} \label{cho2} \beta\big(\alpha,\Gamma^d\big) \leq \left\lbrace
\begin{array}{lll}  \alpha &\text{ if }  0 < \alpha \leq d-2 &\\
  \frac{d}{2} + \frac{\alpha}{2} - 1  &\text{ if }  d-2 < \alpha \leq d  & \text{(Cho, Ham and Lee \cite{cho}).}\\
\alpha-1 &\text{ if } d \leq \alpha \leq d+1 & \end{array} \right. \end{equation}
When $d=3$, these bounds are equal and give the exact value of $\beta\big(\alpha,\Gamma^3\big)$. The exact value of $\beta\big(\alpha,\Gamma^2\big)$ was found by Erdoğan in \cite{erdogan2} (Sjölin's lower bound $\beta\big(\alpha,\Gamma^d\big) \geq \alpha-1$ is a straightforward consequence of Parseval and duality). In Section \ref{conicalaverages}, it will be shown that
\begin{equation} \label{newdecay} \beta\big(\alpha,\Gamma^d\big) \geq \alpha -1 + \frac{d-\alpha}{d+1}, \quad \alpha \in \left(0,d+1\right], \end{equation}
improving \eqref{cho} in the range $\frac{d+1}{2}+\frac{2}{d-1} < \alpha < d$, which is nonempty only for $d \geq 4$. 

As in \cite{du}, the inequality \eqref{newdecay} will be deduced from a linear refined Strichartz inequality for the wave equation, the proof of which is given in Section \ref{linearstrichartz}. The difference between the wave and Schrödinger case is that now the boxes at different scales in the theorem statement no longer intersect in a clean way; the novelty of this work consists in dealing with this issue. In Section \ref{multilinearstrichartz}, a multilinear refined Strichartz inequality will be deduced from the linear one, following \cite{du2}. By using the multilinear version, the bound \eqref{newdecay} could possibly be improved further with the techniques from \cite{zhang}; the known values of $\beta\big(\alpha,\Gamma^d\big)$ for the $d$-dimensional cone tend to mimic the known values of $\beta\big(\alpha, S^{d-1}\big)$ for the sphere of one dimension less. 

Using a counterexample, the upper bound 
\begin{equation} \label{newdecay2} \beta\big(\alpha,\Gamma^d\big) \leq \alpha -1 + \frac{2(d+1-\alpha)}{d+1}, \quad \alpha \in \left(0,d+1\right), \quad d \geq 5. \end{equation}
will also be proved in Section \ref{conicalaverages}, this is the analogy to the upper bound for the spherical averages, due to Luca and Rogers \cite{luca}. This improves \eqref{cho2} in the range $\frac{d+1}{2} < \alpha < d-\frac{4}{d-3}$, this interval being nonempty when $d \geq 6$.

For $q \in [1,\infty]$, let $s_d(\alpha,q)$ be the infimum over all $s$ satisfying the inequality
\begin{equation} \label{strichartz} \lVert u\rVert _{L^q(\mu)} \lesssim C(\mu) \left(\left\lVert u\left( \cdot, 0\right)\right\rVert _{H^s} + \left\lVert  u_t\left(\cdot, 0 \right) \right\rVert _{H^{s-1}} \right) \quad \text{ for } \quad \Delta u - u_{tt} = 0, \end{equation}
where $H^s$ refers to the inhomogeneous Sobolev space with norm
\[ \lVert f\rVert _{H^s} := \left( \int \left\lvert  \widehat{f}(\xi) \right\rvert^2 \left( 1+ \lvert \xi\rvert^2 \right)^s \, d\xi \right)^{1/2}, \]
and
\[C(\mu) := \begin{cases} \left\lVert \mu\right\rVert ^{\frac{1}{q}-\frac{1}{2}} c_{\alpha}(\mu)^{\frac{1}{2}} &\text{ if } 1 \leq q \leq 2 \\
c_{\alpha}(\mu)^{\frac{1}{q}} &\text{ if } 2 < q \leq \infty. \end{cases} \]
The quantity $s_d(\alpha,2)$ is related to $\beta\big(\alpha,\Gamma^d\big)$ through the equation 
\begin{equation} \label{beta-s} \beta\big(\alpha,\Gamma^d\big)= d-2s_d(\alpha,2), \end{equation}
see \cite{wolff,erdogan2,cho,rogers}; a proof is given here in Proposition~\ref{relationship}. The piecewise definition of $C(\mu)$ is what makes the proof of \eqref{beta-s} work, although for known counterexamples the choice of scaling does not seem to matter (see e.g.~\cite{rogers}). For $d \geq 3$, the corresponding best known bounds for $s_d(\alpha,q)$ are
\begin{equation} \label{qsmall} s(\alpha,2,d) \leq s_d(\alpha,q) \leq \widetilde{s}(\alpha,2,d) \quad \text{for $q \in [1,2]$,} \end{equation}
and
\begin{equation} \label{qlarge} s(\alpha,q,d) \leq s_d(\alpha,q) \leq \widetilde{s}(\alpha,q,d) \quad \text{for $q \in [2,\infty]$,} \end{equation}
where (for $d \geq 2$),
\[ s(\alpha,q,d) := \begin{cases} \max\left( \frac{d}{2}-\frac{\alpha}{q},  \frac{d+1}{4} \right) &\text{ if } 0 < \alpha \leq 1 \\
\max\left( \frac{d}{2}-\frac{\alpha}{q}, \frac{d+1}{4} + \frac{1-\alpha}{2q},  \frac{d+2-\alpha}{4} \right) &\text{ if } 1 < \alpha \leq d  \\
\max\left( \frac{d}{2}-\frac{\alpha}{q}, \frac{d+1}{4} + \frac{d+1-2\alpha}{2q},  \frac{d+1-\alpha}{2} \right) &\text{ if } d < \alpha \leq d+1, \end{cases}  \]
and (for $d \geq 3$),
\[ \widetilde{s}(\alpha,2,d) :=  \begin{cases} \frac{d-\alpha}{2} &\text{ if } 0 < \alpha \leq \frac{d-1}{2} \\
 \frac{3d+1}{8} - \frac{\alpha}{4} &\text{ if } \frac{d-1}{2} < \alpha \leq \frac{d+3}{2} \\
\frac{d+1-\alpha}{2}  &\text{ if } \frac{d+3}{2} < \alpha \leq d+1, \end{cases} \]
and for $q>2$, 
\[ \widetilde{s}(\alpha,q,d) :=  \begin{cases} \max\left( \frac{d}{2}-\frac{\alpha}{q},  \frac{d+1}{4}, \frac{3d+1}{8} - \frac{\alpha}{4} \right) &\text{ if } 0 < \alpha \leq 1 \\
\max\left( \frac{d}{2}-\frac{\alpha}{q}, \frac{d+1}{4} + \frac{1-\alpha}{2q}, \frac{3d+1}{8} - \frac{\alpha}{4} \right) &\text{ if } 1 < \alpha \leq d \\
\max\left( \frac{d}{2}-\frac{\alpha}{q}, \frac{d+1}{4} + \frac{d+1-2\alpha}{2q},  \frac{d+1-\alpha}{2} \right) &\text{ if } d < \alpha \leq d+1. \end{cases} \]
The piecewise intervals in $\widetilde{s}(\alpha,q,d)$ are only subdivided this way for simplicity, and for $d> 3$ could be made slightly more optimal as in $\widetilde{s}(\alpha,2,d)$. The lower bound in \eqref{qlarge} is from \cite{cho}, and the lower bound in \eqref{qsmall} is from \cite{rogers} (and works even if the constant $\lVert \mu\rVert ^{\frac{1}{q}-\frac{1}{2}}c_{\alpha}(\mu)^{\frac{1}{2}}$ is relaxed to $c_{\alpha}(\mu)^{\frac{1}{q}}$). The upper bound in \eqref{qsmall} and \eqref{qlarge} is from \cite{cho}. Earlier results were obtained by Oberlin in \cite{oberlin}. The inequalities \eqref{qsmall} and \eqref{qlarge} determine $s_3(\alpha,q) = s(\alpha,q,3)$ for all $q$. The equality $s_2(\alpha,q) = s(\alpha,q,2)$ holds for all $q$; this was determined for $\alpha \geq 1$ in \cite{erdogan2}, and for $\alpha \leq 1$ in \cite[Eq. 90]{wolff2} and \cite[Sect. 4]{rogers}.

The relation \eqref{beta-s} combined with the bound for $\beta\big(\alpha,\Gamma^d\big)$ in \eqref{newdecay} yields
\[ s_d(\alpha,2) \leq \frac{d}{2} - \frac{1}{2} \left(\alpha-1+ \frac{d-\alpha}{d+1} \right). \]
The range of improvement is the same as for $\beta\big(\alpha,\Gamma^d\big)$, given by $\frac{d+1}{2}+\frac{2}{d-1} < \alpha < d$. There would also be some small improvement for slightly larger $q$ by interpolation of $H^s$ spaces. For $q \geq 4$ the known inequalities are already optimal, since the lower and upper bounds in \eqref{qlarge} are the same. 

The upper bound in \eqref{newdecay2} gives 
\[ s_d(\alpha,2) \geq \frac{d}{2} - \frac{1}{2} \left(  \alpha -1 + \frac{2(d+1-\alpha)}{d+1} \right), \quad \alpha \in \left(0,d+1\right), \quad d \geq 5, \]
which improves the lower bound in \eqref{qsmall} for $q=2$, $d \geq 6$ in the range $\frac{d+1}{2} < \alpha < d-\frac{4}{d-3}$. This shows that $s(\alpha,q,d)$ is not the optimal value for all $d$, which disproves a conjecture from \cite{cho}.

The Strichartz inequality \eqref{strichartz} was also considered in \cite{rogers} for a restricted class of measures of the form $\nu = \mu \otimes \lambda$, where $\mu$ is a compactly supported $\alpha$-dimensional measure in the unit ball of $\mathbb{R}^d$ and $d\lambda = \chi_{[0,1]} \, dm$, where $m$ is the Lebesgue measure on $\mathbb{R}$. For this restricted class, the infimum over $s$ in \eqref{strichartz} was shown in \cite{rogers} to be related to Falconer's distance set problem. Unlike in \eqref{beta-s}, the optimal value of $s$ has no relationship to the optimal $L^2$ decay of conical averages, since for this restricted class of measures the (very strong) decay
\[ \int \left\lvert  \widehat{\nu}(R \xi) \right\rvert^2 \, d\sigma_{\Gamma}(\xi) \lesssim I_{\alpha}(\mu) R^{-(\alpha+2)} \quad \text{ for all } R >0, \]
follows directly from a simple computation (Proposition~\ref{trivial}) involving the characterisation of the energy $I_{\alpha}(\mu)$ as
\begin{equation} \label{energydef} I_{\alpha}(\mu) := \int \lvert x-y\rvert^{-\alpha} \, d\mu(x) \, d\mu(y) =c_{\alpha,d} \int \lvert \xi\rvert^{\alpha-d} \left\lvert \widehat{\mu}(\xi)\right\rvert^2 \, d\xi, \quad 0 < \alpha < d. \end{equation}

\subsection{Notation} 

Throughout, let $A: \mathbb{R}^{d+1} \to \mathbb{R}^{d+1}$ be the unitary operator defined through the standard basis by
\begin{equation} \label{unitary} e_i \mapsto e_i \quad \text{for} \quad 1 \leq i \leq d-1, \quad \frac{ e_{d+1}+e_d}{\sqrt{2}} \mapsto e_d, \quad \frac{ e_{d+1}-e_d}{\sqrt{2}} \mapsto e_{d+1}, \end{equation}
and let $E$ be the extension operator for the truncated cone, defined by
\[ Ef(x,t) = \int_{ B(0,2) \setminus B(0,1)}  e^{2\pi i(\langle \xi, x \rangle + \lvert \xi\rvert t)} f(\xi) \, d\xi. \] 

Let $m$ be the Lebesgue measure on Euclidean space. When used as a function the symbol $\pi$ will mean the natural projection from $\mathbb{R}^{d+1}$ to $\mathbb{R}^d$.


For non-negative $X$ and $Y$ the notation $X \lessapprox Y$ will mean $X \leq C_{\epsilon} R^{\epsilon} Y$ where $\epsilon$ is arbitrarily small. Similarly $X \approx Y$ means $X \lessapprox Y$ and $Y \lessapprox X$. For any box $X$ and constant $C \geq 0$, the set $CX$ will the box with the same centre, but with side lengths scaled by $C$.

\section{Modified wave packet decomposition for the cone}
The wave packet decomposition used here is based on those in \cite{guth,cho,li,shayya,ou}. In the decomposition, the truncated cone is partitioned into caps; to construct these, partition the $(d-1)$-dimensional sphere $S^{d-1}$ into spherical caps $C$ of diameter $\delta^{1/2}$, where $\delta \in (0,1)$, and define each $\tau=\tau_C$ by
\[ \tau = \left\{ \left(\xi, \lvert \xi\rvert\right) \in \Gamma: \frac{ \xi}{\lvert \xi\rvert} \in C \right\}. \]
The set $\tau$ is called a cap at scale $\delta^{1/2}$. It is contained in a box of dimensions
\[ \sim \delta^{1/2} \times \dotsm \times \delta^{1/2} \times 1 \times \delta, \]
where the second last coordinate refers to the flat direction in $\tau$, and the last direction is normal to the cone at $\tau$. The normal to $\tau$ is defined to be the normal to the cone at the barycentre of $\tau$ with respect to $\sigma_{\Gamma}$. Similarly, the flat direction of $\tau$ is the barycentre of $C$ (in either case the use of the barycentre is not important, any other point in $\tau$ or $C$ would work).

The dimensions of the boxes in the following wave packet decomposition are adapted to a Lorentz rescaling argument used later, which causes them to differ from the usual case in which the boxes in frequency space are dual to those in physical space (see e.g. \cite{ou}).

\begin{proposition} \label{primitivewave} Fix $\delta>0$. Let $B$ be a box in $\mathbb{R}^d$ of dimensions 
\[ R^{-1/4} \times \dotsm \times R^{-1/4} \times 1, \]
with sides parallel to the coordinate axes, and let $\mathcal{D}$ be a collection of boxes $D \subseteq \mathbb{R}^d$ of dimensions
\[ R^{3/4+ \delta} \times \dotsm \times R^{3/4+\delta} \times R^{1/2+\delta}, \]
also with sides parallel to the coordinate axes, such that the boxes $(1/2)D$ form a finitely overlapping cover of $\mathbb{R}^d$. Suppose that $\Phi \in C^{\infty}((4/3) B, \mathbb{R})$ satisfies
\begin{equation} \label{derivative} \text{$\lvert \partial_i^k \Phi \rvert \lesssim_k R^{3k/4}$ for $1 \leq i \leq d-1$}  \quad \text{and} \quad \text{$\lvert \partial_d^k \Phi \rvert \lesssim_k R^{k/2}$ for all $k \geq 1$.}  \end{equation}
Then for sufficiently large $R$, any $f \in L^2(\mathbb{R}^d)$ supported in $B$ can be decomposed as $f = \sum_D f_D$, such that:
\begin{enumerate} \item Each $f_D$ is supported in $(4/3)B$.
\item The $f_D$'s satisfy the Bessel-type inequality
 \begin{equation} \label{essorthogonality} \sum_D \left\lVert  f_D \right\rVert _2^2 \lesssim \lVert f\rVert _2^2. \end{equation}
\item Each $\widecheck{\Phi f_D}$ is essentially supported in $D$, meaning that for any $N>0$ and $z \notin D$,
 \[ \left\lvert \widecheck{ \Phi f_D}(z)\right\rvert \lesssim_N R^{-N} \lVert f\rVert _2. \] 
\item For any $N>0$, \[ \sum_{D: z \notin D} \left\lvert \widecheck{ \Phi f_D}(z)\right\rvert \lesssim_N R^{-N} \lVert f\rVert _2 \quad \text{for any $z \in \mathbb{R}^d$.} \] \end{enumerate}
\end{proposition}

\begin{proof} By rescaling a smooth bump function on the unit cube, there exists a smooth function $\psi$ with $0 \leq \psi \leq 1$, such that $\psi$ is equal to 1 on $B$, vanishes outside $(5/4)B$ and satisfies
\begin{equation} \label{scaling} \text{$\left\lvert  \partial_i^k \psi \right\rvert \lesssim_k R^{k/4}$ for $1 \leq i \leq d-1$} \quad \text{and} \quad \text{$\left\lvert  \partial_d^k \psi \right\rvert \lesssim_k 1$ for $k \geq 1$.} \end{equation}
Let $\{\phi_D\}_D$ be a smooth partition of unity subordinate to the cover of $\mathbb{R}^d$ by the sets $(3/4)D$. Then
\[ f= \psi f = \sum_D \psi\left(\widehat{\phi_D} \ast f\right) =: \sum_D f_D, \]
where the convergence holds in $L^2(\mathbb{R}^d)$ by Plancherel, and each $f_D$ is supported in $(4/3)B$ by the support condition on $\psi$. By Plancherel, the $L^2$ norm of $f$ satisfies
\[ \lVert f\rVert _2^2 = \sum_{i,j}  \left\langle \phi_{D_i} \widecheck{f} ,\phi_{D_j} \widecheck{f} \right\rangle \\
\geq \sum_D \left\lVert  \phi_{D} \widecheck{f} \right\rVert _2^2 = \sum_D \left\lVert  \widehat{\phi_{D}} \ast f \right\rVert _2^2 \gtrsim \sum_D \left\lVert  f_D \right\rVert _2^2, \]
which proves \eqref{essorthogonality}. 

It remains to check the support conditions on each $\widecheck{\Phi f_D}$. Let $w = \Phi \psi$, so that by \eqref{derivative} and \eqref{scaling},
\[ \text{$\left\lvert \partial_i^k w \right\rvert \lesssim_k R^{3k/4}$ for $1 \leq i \leq d-1$} \quad \text{and} \quad \text{$\left\lvert \partial_d^k w\right\rvert \lesssim_k R^{k/2}$ for $k \geq 1$.} \] By integrating by parts $k$ times, 
\[ \text{$\widecheck{w}(z) \lesssim_k \frac{m(B)}{(\lvert z_i\rvert R^{-3/4} )^k}$ for $1 \leq i \leq d-1$} \quad \text{and} \quad \text{$\widecheck{w}(z) \lesssim_k \frac{m(B)}{(\lvert z_d\rvert R^{-1/2} )^k}$ for $z \in \mathbb{R}^d$.} \]
Hence for $z \notin D$, 
\begin{align*} \left\lvert \widecheck{ \Phi f_D}(z) \right\rvert &= \left\lvert \widecheck{ w} \ast \left(\phi_D \widecheck{f}\right)(z)\right\rvert  \\
&\leq \lVert f\rVert _1 \int_{(3/4)D}  \left\lvert \widecheck{w}(z-y)\right\rvert \phi_D(y)  \, dy \\
&\lesssim_N R^{-N} \lVert f\rVert _2,\end{align*}
for any $N$, by taking $k$ large enough depending on $\delta$. This shows that $\widecheck{ \Phi f_D}$ is essentially supported in $D$. Similarly, since the cover is finitely overlapping, applying similar working and summing a geometric series gives
\[ \sum_{D: z \notin D} \left\lvert \widecheck{ \Phi f_D}(z)\right\rvert \lesssim_N R^{-N} \lVert f\rVert _2 \quad \text{for any $z \in \mathbb{R}^d$,} \]
which proves the proposition. \end{proof}

\begin{proposition} \label{wavepacket} Let $\tau$ be a cap in $\Gamma^d$ at scale $R^{-1/4}$. Then any $f \in L^2(\pi(\tau))$ can be decomposed as $f= \sum_T f_T$ with
\begin{equation} \label{essorthogonality2} \sum_T \left\lVert  f_T \right\rVert _2^2 \lesssim \lVert f\rVert _2^2, \end{equation}
such that each $f_T$ has support in $(4/3) \pi(\tau)$, where the $T$'s are boxes with dimensions
\[ \sim R^{3/4+\delta} \times \dotsm \times R^{3/4+ \delta} \times R^{1/2+\delta} \times R, \]
with long axis normal to $\tau$ and short axis in the flat direction of $\tau$, and such that
\begin{equation} \label{esssupport} \sum_{T: (x,t) \notin T} \left\lvert  Ef_T(x,t) \right\rvert \lesssim_N R^{-N} \lVert f\rVert _2 \quad \text{for} \quad \lvert (x,t)\rvert \leq R. \end{equation}
\end{proposition}

\begin{proof} Since the condition $\lvert (x,t)\rvert \leq R$ is rotation invariant, after applying a rotation of $\mathbb{R}^{d+1}$ which fixes the cone it may be assumed that the flat direction of $\tau$ is $\frac{e_{d+1}+e_d}{\sqrt{2}}$. After then applying the unitary $A$ (defined in \eqref{unitary}), it may be assumed that the normal to $\tau$ is completely in the $e_{d+1}$-direction, and that $(4/3)\tau$ is contained in the graph of the function
\[ h(\omega_1, \dotsc, \omega_{d}) := \frac{ \sum_{i=1}^{d-1}  \omega_i^2}{2 \omega_d}, \]
restricted to $(4/3)B$ where
\[ B:= \left[ -R^{-1/4}, R^{-1/4} \right] \times \dotsm \times \left[ -R^{-1/4}, R^{-1/4} \right] \times \left[ \sqrt{2}, 2 \sqrt{2} \right] \subseteq \mathbb{R}^d. \]
The extension operator corresponding to $h$ will still be denoted by $E$. Let $f_T = f_D$ from Proposition~\ref{primitivewave} where each $T= D \times \mathbb{R}$. The inequality \eqref{essorthogonality2} and the support condition on the $f_T$'s both follow from Proposition~\ref{primitivewave}. Moreover,
\[ Ef_T(x,t) =  \int e^{2\pi i \left( \langle \omega, x \rangle + h(\omega) t \right) } f_T(\omega) \, d\omega = \widecheck{ \Phi_t f_D }(x), \]
where $\Phi_t(\omega) := e^{2\pi ih(\omega)t}$ restricted to $(4/3) B$, and from the definition of $h$,
\[ \lvert \partial_i^k \Phi_t \rvert(\omega) \lesssim_k R^{3k/4} \quad \text{for} \quad 1 \leq i \leq d-1 \quad \text{and} \quad \lvert \partial_d^k \Phi_t \rvert(\omega) \lesssim_k R^{k/2}, \]
for all $k \geq 1$, uniformly for $\lvert t\rvert \leq R$ and $\omega \in (4/3)B$. Inequality \eqref{esssupport} then follows from Proposition~\ref{primitivewave}. This finishes the proof.
\end{proof}

\section{Linear refined Strichartz inequality} \label{linearstrichartz}
Most of the notation used throughout this section will be similar to that in \cite{li}, to emphasise the analogy with the Schrödinger case. The proof of the linear refined Strichartz inequality will use the following (sharp) decoupling theorem for the truncated cone from \cite{bourgain}. 
\begin{theorem}[{\cite[Theorem~1.2]{bourgain}}]  \label{decoupling} Let $q = \frac{2(d+1)}{d-1}$ and $\delta \in (0,1)$, partition $\Gamma^d$ into caps $\tau$ at scale $\delta^{1/2}$, and let $F = \sum_{\tau} F_{\tau}$ be a function on $\mathbb{R}^{d+1}$ such that for each $\tau$ the support of $\widehat{F_{\tau}}$ is contained in the $\delta$-neighbourhood of $\tau$. Then for any $\epsilon >0$,
\[ \lVert F\rVert _q \leq C_{\epsilon} \delta^{-\epsilon} \left(\sum_{\tau } \lVert F_{\tau}\rVert _q^2 \right)^{1/2}. \] \end{theorem}
The next theorem is the main result of this section. By duality, pigeonholing, the uncertainty principle and by Hölder's inequality, this theorem will imply an $L^2$-Strichartz inequality for fractal measures (see Corollary~\ref{averages} and Proposition~\ref{relationship}). 
\begin{theorem} \label{strichartzr} Fix $d \geq 2$, let $\gamma = \frac{1}{2} - \frac{1}{q}$ and $q  = \frac{2(d+1)}{d-1}$. Suppose that $f \in L^2(\mathbb{R}^d)$ is supported in $\{ \xi_d \geq 0\} \cap B(0,2) \setminus B(0,1)$, and let $Y= \bigcup Q$ be a collection of $R^{1/2} A^*\mathbb{Z}^{d+1}$-lattice cubes inside a ball of radius $R$. If $\lVert Ef\rVert _{L^q(Q)}$ is constant up to a factor of 2 as $Q$ varies over $Y$, and if the cubes are arranged in slabs of the form $A^*\left(\mathbb{R}^{d} \times [R^{1/2} j, R^{1/2} j + R^{1/2}]\right)$ with $j \in \mathbb{Z}$, such that each slab intersecting $Y$ intersects $\sim \sigma$ cubes in $Y$, then for any $\epsilon >0$,
\[ \lVert  Ef\rVert _{L^{q}(Y)} \leq C_{\epsilon}  R^{\epsilon} \sigma^{-\gamma}\lVert f\rVert _2. \]
\end{theorem}

\begin{remark} \label{scales} To use induction, it will be easier to prove a slightly more general statement, where each $Q$ in the theorem is replaced by $CQ$, for any constant $C \in [1,2]$, but the slabs are unchanged. The inductive assumption will be that this slightly more general statement of the theorem holds whenever $R$ is replaced everywhere in the theorem by $\widetilde{R}$, for any $\widetilde{R} \leq R^{3/4}$ ($3/4$ is not important, any exponent in $(1/2,1)$ would work). 

Given $\epsilon >0$, the inductive assumption is that 
\[ \lVert  Ef\rVert _{L^{q}(Y)} \leq C_{\epsilon} \widetilde{R}^{\epsilon}\sigma^{-\gamma} \lVert f\rVert _2, \]
 for any $f$ satisfying the (generalised) assumptions of the theorem, and any $\widetilde{R} \leq R^{3/4}$. Using this, it will be shown that for any $\delta>0$, the inequality 
\[ \lVert  Ef\rVert _{L^{q}(Y)} \leq C_{\epsilon} C_{\delta} R^{O(\delta)} R^{\epsilon/2} \sigma^{-\gamma}\lVert f\rVert _2, \]
holds for any $f$ satisfying the (generalised) assumptions of the theorem. By choosing $\delta>0$ small enough, depending only on $\epsilon$, and then taking $R$ large enough (depending only on $\epsilon$), this will give
\[ \lVert  Ef\rVert _{L^{q}(Y)} \leq C_{\epsilon} R^{\epsilon} \sigma^{-\gamma}\lVert f\rVert _2, \]
with the same constant $C_{\epsilon}$, which will close the induction and prove the theorem. To simplify notation the $R^{\epsilon}$ and $R^{\delta}$ factors will be absorbed into the symbols $\approx$ and $\lessapprox$, and therefore this argument will not be carried out explicitly. \end{remark}

\begin{proof}[Proof of Theorem~\ref{strichartzr}] After replacing $f(\xi)$ by $e^{-2\pi i(\langle x_0, \xi \rangle + t_0 \lvert \xi\rvert)} f(\xi)$ it may be assumed that the ball of radius $R$ containing $Y$ is centred at the origin. 

The function $f$ will first be broken up to make use of the inductive assumption that the theorems holds at smaller scales than $R$. Fix $\delta>0$, and use Proposition~\ref{wavepacket} to decompose 
\[ f = \sum_{\tau,\Box_{\tau}} f_{\Box}, \quad Ef = \sum_{\tau,\Box_{\tau}} Ef_{\Box}, \]
where the caps $\tau$ partitioning the truncated half cone $\Gamma_+$ are at scale $\sim R^{-1/4}$, each contained in a corresponding box with dimensions 
\[ \sim R^{-1/4} \times \dotsm \times R^{-1/4} \times 1 \times R^{-1/2}, \]
and each $Ef_{\Box}$ has (distributional) Fourier transform supported in the cap $\tau$ corresponding to $\Box$. Each $Ef_{\Box}$ satisfies $\left\lvert Ef_{\Box}(z,t) \right\rvert \lesssim_N R^{-N}$ for $(z,t) \in B(0,R)$ outside a set $\Box$ of dimensions 
\[ \sim R^{3/4+\delta} \times \dotsm \times R^{3/4+\delta} \times R^{1/2+\delta} \times R, \]
which has long axis normal to $\tau$ and short axis in the flat direction of $\tau$. Although each set $\Box$ depends on a cap $\tau$, this will be suppressed in the notation. Each $Ef_{\Box}$ can be partitioned further
\begin{equation} \label{tail} Ef_{\Box} = \sum_{2S \cap \Box \neq \emptyset} \eta_S Ef_{\Box} +\tail \quad \text{in} \quad B(0,R), \end{equation}
where $\tail(z,t) \lesssim_N R^{-N} \|f\|_2$ for $(z,t) \in B(0,R)$. For each $\tau$, the sets $S$ partition physical space and have the same axis orientations as $\Box$, with dimensions
\[ R^{1/2+\delta} \times \dotsm \times R^{1/2+\delta} \times R^{1/4+\delta} \times R^{3/4+\delta}. \]
These sets $S$ will become the cubes in the theorem statement after a Lorentz rescaling. The functions $\eta_S$ are smooth and sum to 1, with each $\eta_S \lesssim 1$ on $S$, non-negative, $\eta_S \lesssim_N R^{-N}$ outside $2S \cap \mathcal{N}_{R^{1/2+\delta}}(S)$, such that $\widehat{\eta_S}$ is supported in a box centred at the origin of dimensions
\begin{equation} \label{poisson} \sim R^{-1/2} \times \dotsm \times R^{-1/2} \times R^{-1/4} \times R^{-1/2}, \end{equation}
with long axis in the flat direction of $\tau$. Such a partition can be constructed by using the Poisson summation formula on the integer lattice to get a smooth non-negative function $\eta$ with 
\[ \sum_{k \in \mathbb{Z}^{d+1}} \eta(x-k) = 1 \quad \text{for all $x \in \mathbb{R}^{d+1}$,} \quad \text{and} \quad \supp \widehat{\eta} \subseteq B(0,c_d).\]
Rescaling by the dimensions in \eqref{poisson} and then grouping the functions together corresponding to scaled lattice points in $S$ gives the required functions $\eta_S$. The steep decay of $\eta_S$ outside the $\approx R^{1/2}$ neigbourhood of $S$ will be important when intersecting $S$ with the $R^{1/2}$-cubes in the theorem statement.  

Through pigeonholing, the sets $S$ will now be organised into collections that satisfy the theorem assumptions after a Lorentz rescaling. For each fixed $\Box$, sort the boxes $S$ from \eqref{tail} into sets $\mathbb{S}_{\lambda}$ according to the dyadic value $\lambda$ of $\left\lVert Ef_{\Box}\right\rVert _{L^q(2S)}$. By ignoring the very small values of $\lambda$ which do not contribute significantly, there are $\lesssim \log R$ relevant values of $\lambda$, which will be the only ones considered from now on. For each fixed $\lambda$ and dyadic number $\sigma_{\Box}$, define $\mathbb{S}_{\lambda,\sigma_{\Box}}$ by $S \in \mathbb{S}_{\lambda, \sigma_{\Box}}$ if and only if the number of boxes in $\mathbb{S}_{\lambda}$ inside the slab of width $R^{3/4+\delta}$ parallel to the tangent plane at $\tau$ lies in $\left[\sigma_{\Box}, 2\sigma_{\Box} \right)$. These slabs are different to those in the theorem statement, but will be rotated and scaled to use the inductive assumption. There are $\lesssim \log R$ dyadic values of $\sigma_{\Box}$, so there are $\lesssim \left(\log R\right)^2$ relevant pairs $\left(\lambda, \sigma_{\Box}\right)$. For any set $\Box$ let $Y_{\Box,\lambda,\sigma_{\Box}}$ be the union of boxes $S$ in $\mathbb{S}_{\lambda,\sigma_{\Box}}$. By \eqref{tail},
\[ Ef = \sum_{\lambda, \sigma_{\Box}} \sum_{\Box} \eta_{Y_{\Box,\lambda,\sigma_{\Box}}}Ef_{\Box} + \tail \quad \text{in $B(0,R)$,} \quad \eta_{Y_{\Box,\lambda,\sigma_{\Box}}} := \sum_{S \subseteq Y_{\Box,\lambda,\sigma_{\Box}}} \eta_S. \]
By the triangle inequality and the standard pigeonhole principle, there is a fixed pair $(\lambda, \sigma_{\Box})$ such that 
\begin{equation} \label{pigeon} \lVert Ef\rVert _{L^{q}(CQ)} \lessapprox \left\lVert   \sum_{\Box} \eta_{Y_{\Box,\lambda,\sigma_{\Box}}} Ef_{\Box} \right\rVert _{L^{q}(CQ)} + R^{-N} \lVert f\rVert _2,\end{equation}
for a fraction $\approx 1$ of the cubes $Q$. The quantity $R^{-N}\lVert f\rVert _2$ is a remainder term which absorbs the very small values of $\lambda$ and the tails from the wave packet decomposition, and it may essentially be ignored in the inequalities that follow. Henceforth $Y_{\Box,\lambda,\sigma_{\Box}}$ will be abbreviated to $Y_{\Box}$. Each $\Box$ can be sorted according to the dyadic value of $\lVert f_{\Box}\rVert _2$, and since there are only $\sim \log R$ relevant dyadic values, the triangle inequality and the standard pigeonhole principle applied again to the remaining cubes satisfying \eqref{pigeon} together yield a subset $\mathbb{B}$ of sets $\Box$ such that $\lVert f_{\Box}\rVert _2$ is constant over $\mathbb{B}$ up to a factor of 2, and 
\begin{equation} \label{pigeonhole} \lVert Ef\rVert _{L^{q}(CQ)} \lessapprox \left\lVert   \sum_{\Box \in \mathbb{B}} \eta_{Y_{\Box}}Ef_{\Box} \right\rVert _{L^{q}(CQ)} + R^{-N} \lVert f\rVert _2, \end{equation}
again for a fraction $\approx 1$ of the cubes $Q$. By further pigeonholing the remaining cubes satisfying \eqref{pigeonhole}, there is a dyadic number $\mu$ and a fraction $\approx 1$ of the cubes $CQ$, with union $Y'$, such that the bound in \eqref{pigeonhole} still holds and for each $CQ \subseteq Y'$ the cube $R^{2\delta}Q$ intersects $\sim \mu$ of the sets $Y_{\Box}$ with $\Box \in \mathbb{B}$. 

 In order to apply the inductive assumption to each $Ef_{\Box}$ individually, the decoupling theorem will be used to decouple the right hand side of \eqref{pigeonhole}. To set this up, let $\eta_Q$ be smooth non-negative functions such that $\eta_Q \sim 1$ on $2Q$, $\eta_Q$ is rapidly decaying outside $2Q$, with each $\widehat{\eta_Q}$ compactly supported in a ball of radius $\sim R^{-1/2}$, and such that $\sum_Q \eta_Q \lesssim 1$ (this can be done, for example, by using Poisson summation again). To apply decoupling, each function $\eta_{Y_{\Box}} \eta_{Q}Ef_{\Box}$ has Fourier transform supported in an $\sim R^{-1/2}$ neighbourhood of $\tau$ since the short direction of each $S\subseteq Y_{\Box}$ is in the flat direction of $\tau$. Moreover, for each fixed $\tau$, the sets $\Box$ corresponding to $\tau$ form a finitely overlapping cover of physical space and have side lengths at least as large as those of $Q$, which means that each cube $R^{2\delta}Q$ intersects $\lessapprox 1$ set $\Box$ associated to $\tau$, for each $\tau$. Therefore the disjointness assumption in the decoupling theorem is satisfied, and so by applying Theorem~\ref{decoupling} to each $CQ \subseteq Y'$,
\begin{align*} \lVert Ef\rVert _{L^{q}(CQ)} &\lessapprox \left\lVert   \sum_{\Box \in \mathbb{B}} \eta_{Y_{\Box}}Ef_{\Box}\right\rVert _{L^{q}(CQ)} + R^{-N} \lVert f\rVert _2  && \text{(by \eqref{pigeonhole})} \\
&\lessapprox \Bigg\lVert   \sum_{\substack{\Box \in \mathbb{B}  \\
Y_{\Box} \cap R^{2\delta}Q \neq \emptyset }}\eta_{Y_{\Box}} \eta_{Q} Ef_{\Box} \Bigg\rVert _{q} + R^{-N} \lVert f\rVert _2  \\
&\lessapprox \Bigg( \sum_{\substack{\Box \in \mathbb{B} \\
Y_{\Box} \cap R^{2\delta}Q \neq \emptyset }}\left\lVert    \eta_{Y_{\Box}} \eta_{Q}Ef_{\Box} \right\rVert _{q}^2 \Bigg)^{1/2}  + R^{-N} \lVert f\rVert _2 && \text{(by Theorem~\ref{decoupling})}  \\
&\lesssim \mu^{ \frac{1}{2} - \frac{1}{q}}\left( \sum_{\Box \in \mathbb{B} }\left\lVert  \eta_{Y_{\Box}}\eta_{Q} Ef_{\Box} \right\rVert _{q}^{q} \right)^{1/q} + R^{-N} \lVert f\rVert _2 .  \end{align*}
Raising both sides to the power $q$ and summing over $CQ \subseteq Y'$ gives (since $\sum_{Q} \eta_{Q}  \lesssim 1$ and $\eta_{Y_{\Box}}$ is rapidly decaying outside $2Y_{\Box}:= \bigcup_{S \subseteq Y_{\Box}} 2S$)
\begin{align*} \lVert Ef\rVert _{L^{q}(Y')}^{q}  &\lessapprox \mu^{ q \left( \frac{1}{2}- \frac{1}{q}\right) } \sum_{\Box \in \mathbb{B}} \left\lVert  Ef_{\Box}\right\rVert ^{q}_{L^q(2Y_{\Box})} + R^{-Nq} \lVert f\rVert _2^q . \end{align*}
 Since $Y'$ contains a fraction $\approx 1$ of the cubes in $Y$, and since the cubes contribute equally, this gives
\begin{align} \notag \lVert Ef\rVert _{L^{q}(Y)}^{q} &\lessapprox \mu^{ q \left( \frac{1}{2}- \frac{1}{q}\right) } \sum_{\Box \in \mathbb{B}} \left\lVert  Ef_{\Box}\right\rVert ^{q}_{L^q(2Y_{\Box})} + R^{-Nq} \lVert f\rVert _2^q  \\
\label{rescaling} &\lessapprox \mu^{ q \left( \frac{1}{2}- \frac{1}{q}\right) } \sum_{\Box \in \mathbb{B}} \left(  \sigma_{\Box}^{-\gamma}  \lVert f_{\Box}\rVert _2 \right)^q + R^{-Nq} \lVert f\rVert _2^q \quad \text{(see below)}\\
\label{qthroot} &\lesssim \mu^{ q \left( \frac{1}{2}- \frac{1}{q}\right) }\sigma_{\Box}^{-\gamma q}  \lvert \mathbb{B}\rvert^{  1- \frac{q}{2}} \lVert f\rVert _2^{q}, \end{align}
since $\lVert f_{\Box}\rVert _2$ and $\sigma_{\Box}$ are constant in $\Box$ up to a factor of 2, and $\sum_{\Box \in \mathbb{B}} \left\lVert f_{\Box} \right\rVert _2^2 \lesssim \lVert f\rVert _2^2$. To justify the Lorentz rescaling step in \eqref{rescaling}, after a rotation it may be assumed that the flat direction of the cap $\tau$ corresponding to $\Box$ is $\frac{ e_{d+1}+e_d}{\sqrt{2}}$. Define $B: \mathbb{R}^{d+1} \to \mathbb{R}^{d+1}$ by
\[ Bx = \left( R^{1/4}x_1,\dotsc, R^{1/4}x_{d-1}, x_d, R^{1/2}x_{d+1}\right), \] 
and define $\eta$ and $\left(\widetilde{x}, \widetilde{t}\right)$ by 
\[ \left(\eta, \left\lvert \eta\right\rvert\right) = A^*BA (\xi,\lvert \xi\rvert) \quad \text{for $\xi \in \tau$,} \quad \left(\widetilde{x}, \widetilde{t}\, \right) = A^* B^{-1} A(x,t). \]
Under this change of variables, the boxes $2S$ are sent to cubes $2\widetilde{Q}$ of side length $2R^{1/4+\delta}$, whose union is denoted by $\widetilde{Y}$. Moreover,
\[ \lVert Ef_{\Box}\rVert _{L^{q}\left(2Y_{\Box}\right)} = \lVert Eg\rVert _{L^{q}\left(\widetilde{Y}\right)} \quad \text{and} \quad \lVert f_{\Box}\rVert _2 \sim \lVert g\rVert _2, \]
where
\[ g( \eta) := R^{ \frac{d-1}{8}}   \left\lvert \frac{d \xi}{d\eta}\right\rvert f_{\Box} (\xi), \]
which has support contained in 
\begin{equation} \label{stars} \pi\left( \left\{ (\xi,\lvert \xi\rvert): \xi_d \geq 0\right\} \cap B\left(0,3 \sqrt{2}\right) \setminus B(0,1)\right), \end{equation}
provided $\tau$ is at scale $cR^{-1/4}$ for some small constant $c$ depending only on $d$ (which may be assumed). The values $\lVert Ef_{\Box}\rVert _{L^q(2S)} = \lVert Eg\rVert _{L^{q}(2\widetilde{Q})}$ are dyadically constant over $S,\widetilde{Q}$ by definition of the dyadic values $\lambda$. The slabs are of the form $A^* \left( \mathbb{R}^d \times [t_0, t_0 + R^{1/4+\delta}]\right)$ and each contains $\sim \sigma_{\Box}$ cubes, all inside a ball of radius $R^{1/2+2\delta}$. Even though the set in \eqref{stars} may be outside $B(0,2) \setminus B(0,1)$, by a constant linear rescaling this only affects the inequality by a constant factor. Therefore the generalised statement of the theorem at scale $R^{1/2+2\delta}$ can be applied to give
\[ \lVert Ef_{\Box}\rVert _{L^{q}(2Y_{\Box})} \lessapprox \sigma_{\Box}^{-\gamma}\lVert f_{\Box}\rVert _2, \]
which is \eqref{rescaling}. Taking the $q$-th root of \eqref{qthroot} gives
\begin{equation} \label{bound} \lVert Ef\rVert _{L^{q}(Y)} \lessapprox \mu^{ \gamma } \lvert \mathbb{B}\rvert^{ - \gamma} \sigma_{\Box}^{-\gamma}   \lVert f\rVert _2, \quad  \gamma = \frac{1}{2}-\frac{1}{q}. \end{equation}

To finish the proof, it suffices to show that $\mu \lessapprox \frac{ \sigma_{\Box} \lvert \mathbb{B}\rvert}{\sigma}$. The multiplicity $\mu$ satisfies
\begin{align*} \mu m(Y) &\lessapprox \mu m(Y') \\
&\sim  \mu \sum_{CQ \subseteq Y'}  m(Q) \\
&\sim \sum_{CQ \subseteq Y'} \sum_{\substack{\Box \in \mathbb{B} \\
Y_{\Box} \cap R^{2\delta} Q \neq \emptyset }} m( Q) && \text{(by definition of $\mu$)} \\
&\lessapprox \sum_{CQ \subseteq Y'} \sum_{\substack{\Box \in \mathbb{B} \\
Y_{\Box} \cap R^{2\delta} Q \neq \emptyset }} R^{1/4}m\left(Y_{\Box} \cap R^{3\delta} CQ\right) \\
&\lessapprox  \sum_{\Box \in \mathbb{B}} R^{1/4}m\left(Y_{\Box} \cap R^{3\delta}Y'\right) \\
 &\leq \sum_{\Box \in \mathbb{B}} R^{1/4}m\left(Y_{\Box} \cap R^{3\delta} Y\right). \end{align*}
Therefore it suffices to show that for each $\Box \in \mathbb{B}$, 
\[m\left(Y_{\Box} \cap R^{3\delta} Y\right) \lessapprox \frac{ \sigma_{\Box} m(Y) R^{-1/4} }{\sigma}. \]
By breaking both sides of this inequality into slabs $A^*\left(\mathbb{R}^d \times \left[t, t+CR^{1/2 + 3\delta} \right] \right)$ covering $R^{3\delta}Y$ and containing $\approx \sigma$ cubes $Q$ in each slab, it suffices to show that 
\begin{equation} \label{intermediate} m\left(Y_{\Box} \cap A^*\left(\mathbb{R}^d \times \left[t, t+CR^{1/2 + 3\delta} \right] \right)\right) \lessapprox \sigma_{\Box} m(Q) R^{-1/4}, \end{equation}
where $m(Q)= R^{\frac{d+1}{2}}$ depends only on $R$. The long axis of $\Box$ makes an acute angle $\gtrsim 1$ with the $CR^{1/2+3\delta}$ slab, and therefore intersects it in a set of diameter $\lessapprox R^{3/4}$. Hence there are $\lessapprox \sigma_{\Box}$ sets $S \subseteq Y_{\Box}$ in the intersection on the left hand side. Since each such set $S$ has long axis in the same direction as $\Box$, the intersection of $S$ with the $CR^{1/2+3\delta}$ slab is contains in a box of dimensions 
\[  \sim R^{1/2+\delta}  \times \dotsm \times R^{1/2+\delta} \times R^{1/4+\delta} \times R^{1/2+3\delta}, \]
and therefore the intersection has measure $\lessapprox m(Q) R^{-1/4}$. Adding up the contributions of each $S$ gives \eqref{intermediate}, and this yields
\[ \mu \lessapprox \frac{ \sigma_{\Box} }{\sigma} \lvert \mathbb{B}\rvert. \]
Combining with \eqref{bound} results in
\begin{align*} \lVert Ef\rVert _{L^{q}(Y)} &\lessapprox \sigma^{-\gamma}   \lVert f\rVert _2. \end{align*}
By the induction on scales argument explained in Remark \ref{scales}, this proves the theorem. \end{proof}

\begin{remark} A similar example to the Schrödinger case \cite{du2} shows that the inequality is sharp; $Y$ is a collection of $\sigma$ disjoint parallel tubes of radius $R^{1/2}$ and length $R$, $f$ is an essentially orthogonal sum $f= \sum_{\nu} f_{\nu}$ over the different tubes with $\left\lVert f_{\nu}\right\rVert _2$ constant in $\nu$, and each $Ef_{\nu}$ essentially constant and supported on a thinner tube inside the larger one. 

To be more precise, let $\phi$ be a Schwartz function on $\mathbb{R}^d$ satisfying $0 \leq \widehat{\phi} \leq 1$, with $\widehat{\phi}$ equal to 1 on $B((3/2) e_d, \epsilon)$ and equal to zero outside $B((3/2) e_d, 2\epsilon)$, for a fixed small $\epsilon >0$. Then $\left\lvert E\widehat{\phi} \right\rvert \sim 1$ on a ball $B(0,C_{\epsilon})$ for some small $C_{\epsilon}>0$. Let $\tau$ be a cap in the cone at scale $R^{-1/2}$ with centre line in the direction $\frac{e_{d+1}+e_d}{\sqrt{2}}$. Let $\eta = \pi (A^*BA(\xi,\lvert \xi\rvert) ) \in \mathbb{R}^d$ where $B(\omega_1, \ldots, \omega_d, \omega_{d+1}) = (R^{1/2} \omega_1, \dotsc, R^{1/2} \omega_{d-1}, \omega_d, R\omega_{d+1})$, and define $\phi_{\tau}$ on the Fourier side by $\widehat{\phi_{\tau}}(\xi) = \left\lvert \frac{d\eta}{d\xi}\right\rvert \widehat{\phi}(\eta)$. Then $\left\lvert E\widehat{\phi_{\tau}} \right\rvert \sim 1$ on $(A^*BA)B(0,C_{\epsilon})$ by a change of variables, and the restriction of $E\widehat{\phi_{\tau}}$ to $B(0,R)$ is essentially supported on a larger box of similar dimensions $R^{1/2+\delta} \times \dotsm \times R^{1/2+\delta} \times R^{\delta} \times 2R$, by an integration by parts as in the wave packet decomposition. Let $\{T_{\nu}\}$ be a finitely overlapping cover of $B(0,R)$ with translates of this larger box by points in $A^*(\mathbb{Z}^d \times \{0\})$, each containing a corresponding translate $S_{\nu}$ of the smaller box. Let $c(\nu)$ be the centre of $S_{\nu}$. Define $\phi_{\tau,\nu}$ on the Fourier side by $\widehat{\phi_{\tau,\nu}}(\xi) = e^{-2 \pi i \left\langle c(\nu), (\xi,\lvert \xi\rvert) \right\rangle}  \widehat{\phi_{\tau}}(\xi)$, and let $f_{\nu} = \widehat{\phi_{\tau,\nu}}$. Let $\mathbb{T}$ be a subcollection of boxes $T_{\nu}$ which intersect $B(0,R/2)$ and are $R^{1/2+2\delta}$-separated from each other, such that $\lvert \mathbb{T}\rvert = \sigma$. Let $f = \sum_{T_\nu \in \mathbb{T}} f_\nu$. Cover each $S_{\nu}$ with $\sim R^{1/2-2\delta}$ disjoint $R^{1/2+2\delta}A^*\mathbb{Z}^{d+1}$-lattice cubes $Q$, and let $Y$ be the union of all such $Q$, over all $T_{\nu} \in \mathbb{T}$. By a change of variables, for any $Q$ intersecting $S_{\nu}$, 
\[ \lVert Ef\rVert _{L^q(Q)}^q \approx \lVert Ef_{\nu} \rVert _{L^q(Q \cap 2S_{\nu})}^q \approx R^{-1/2} \lVert Ef_{\nu} \rVert _{L^q(S_{\nu})}^q \approx  R^{d/2}. \]
Hence $\lVert Ef\rVert _{L^q(Y)}^q \approx \sigma R^{\frac{d+1}{2}}$. A rotation and integration by parts shows that the $f_{\nu}$ are essentially orthogonal, and so $\lVert f\rVert _2^2 \approx \sigma R^{\frac{d-1}{2}}$, which gives $\lVert Ef\rVert _{L^q(Y)} \gtrapprox \sigma^{-\gamma} \lVert f\rVert _2$. This verifies that the inequality is sharp. \end{remark}

\section{Multilinear refined Strichartz inequality} \label{multilinearstrichartz}
As in the Schrödinger case \cite{du2}, the linear refined Strichartz inequality for the wave equation implies a multilinear version. The proof is similar to the one in \cite{du2}; it utilises the $k$-linear Kakeya inequality in $\mathbb{R}^n$ from \cite{bennett,guth3}. To state this, given $\nu >0$, a collection $\mathbb{T}_1, \dotsc, \mathbb{T}_k$ of $k$ sets consisting of tubes $T_i \in \mathbb{T}_i$, of infinite length and equal radius, is called $\nu$-transverse if
\[ \left\lvert  v_1 \wedge \cdots \wedge v_k \right\rvert \geq \nu \quad \text{for all $T_i \in \mathbb{T}_i$, with $1 \leq i \leq k$,} \]
where $v_i$ is the infinite direction in $T_i$. 
\begin{theorem}[{\cite[Theorem~5.1]{bennett}}] Suppose that $2 \leq k \leq n$ and that $\mathbb{T}_1, \dotsc, \mathbb{T}_k$ are $\nu$-transverse families of tubes in $\mathbb{R}^n$ of radius $\delta$. If $\frac{k}{k-1} < q \leq \infty$, then 
\begin{equation} \label{kakeya} \left\lVert  \prod_{i=1}^k  \left( \sum_{T_i \in \mathbb{T}_i} \chi_{T_i} \right) \right\rVert _{L^{q/k}\left(\mathbb{R}^n\right)} \leq C(\nu,q,n) \prod_{i=1}^k \left( \delta^{n/q} \left\lvert  \mathbb{T}_i \right\rvert \right). \end{equation}
\end{theorem}

In what follows, the preceding theorem will be applied with $n=d+1$. The statement that $f_1, \dotsc, f_k \in L^2(\mathbb{R}^d)$ are transversally supported in $B(0,2) \setminus B(0,1)$ will mean that each $\supp f_i \subseteq B(0,2) \setminus B(0,1)$, and any $k$-tuple $(\xi_1, \dotsc, \xi_k)$ with $\xi_i \in \supp f_i$ satisfies
\[ \left\lvert  G(\xi_1) \wedge \cdots \wedge G(\xi_k) \right\rvert \gtrsim 1, \]
where $G(\xi)$ is the unit normal to the cone at $(\xi,\lvert \xi\rvert)$. 

\begin{theorem} Fix $d \geq 2$. Suppose that $2 \leq k \leq d+1$ and that $f_1, \dotsc, f_k \in L^2(\mathbb{R}^d)$ are transversally supported in $\{\xi_d \geq 0\} \cap B(0,2) \setminus B(0,1)$, and let $Y= \bigcup_{j=1}^N Q_j$ be a collection of $R^{1/2} A^*\mathbb{Z}^{d+1}$-lattice cubes in $B(0,R)$. If for each $i$, $\left\lVert Ef_i\right\rVert _{L^q(Q_j)}$ is constant in $j$ up to a factor of 2, then for any $\epsilon>0$,
\[ \left\lVert  \prod_{i=1}^k \left\lvert Ef_i \right\rvert^{1/k} \right\rVert _{L^{q}(Y)} \leq C_{\epsilon}R^{\epsilon}N^{-\frac{(k-1)\gamma}{k}} \prod_{i=1}^{k} \lVert f_i\rVert _2^{1/k}. \]
where $\gamma = \frac{1}{2}-\frac{1}{q}$ and $q  = \frac{2(d+1)}{d-1}$.
\end{theorem}
\begin{proof} Using Proposition~\ref{wavepacket}, decompose
\[ f_i = \sum f_{\Box,i}, \]
for each $i$, as in the proof of Theorem~\ref{strichartzr}. By Hölder's inequality,
\[ \left\lVert  \prod_{i=1}^k \left\lvert Ef_i \right\rvert^{1/k} \right\rVert _{L^{q}(Y)} \leq \prod_{i=1}^k \left\lVert  Ef_i \right\rVert _{L^{q}(Y)}^{1/k}. \]
To each $f_i$, apply the same steps as in the proof of Theorem~\ref{strichartzr} up to \eqref{bound}, except without replacing the cubes $Q$ by $CQ$, and instead choose the set $Y'$ of cubes in $Y$ to be the same for each $i$ (by doing this first for $f_1$, and then $f_2$ on the good cubes from $f_1$, and so on). The sets $\Box_i$, $Y_{\Box_i}, \mathbb{B}_i$ and parameters $\mu_i$, $\sigma_{\Box,i}$ are all defined as before. Using \eqref{bound} for each $f_i$ (and the generalised version of Theorem~\ref{strichartzr} at \eqref{rescaling} (instead of the inductive assumption)) yields
\begin{equation} \label{aim} \left\lVert  \prod_{i=1}^k \left\lvert Ef_i \right\rvert^{1/k} \right\rVert _{L^{q}(Y)} \lessapprox \left(\prod_{i=1}^k  \frac{\mu_i}{\left\lvert \mathbb{B}_i\right\rvert \sigma_{\Box,i} } \right)^{\gamma/k}  \prod_{i=1}^{k} \lVert f_i\rVert _2^{1/k}. \end{equation}
To finish the proof it therefore suffices to show that
\begin{equation} \label{sufficient} N \prod_{i=1}^k  \mu_i^{\frac{1}{k-1}} \lessapprox \prod_{i=1}^k \left(\left\lvert \mathbb{B}_i\right\rvert \sigma_{\Box,i}\right)^{\frac{1}{k-1}}. \end{equation}
To this end, for each $Y_{\Box_i}$ let $Z_{\Box_i} = \bigcup_{S \subseteq Y_{\Box_i}} F(S)$, where for each box $S \subseteq Y_{\Box_i}$ of dimensions 
\[ R^{1/2+ \delta} \times \dotsm \times R^{1/2+\delta} \times R^{1/4+\delta} \times R^{3/4+\delta}, \]
the set $F(S)$ is defined to be the box with dimensions
\[ R^{1/2+ 4\delta} \times \dotsm \times R^{1/2+4\delta} \times R^{1/2+4\delta} \times R^{3/4+4\delta}, \]
with the same centre and orientations as $S$, obtained by expanding $S$ by $\approx R^{1/4}$ in the second last direction and $\approx 1$ in all the other directions. This adjustment ensures that if the expanded cube $R^{2\delta}Q$ of side length $R^{1/2 + 2\delta}$ intersects $Y_{\Box_i}$, then $Q$ is completely contained in $Z_{\Box_i}$. Recall that each $\mu_i$ was chosen so that for each $Q \subseteq Y'$, there are $\sim \mu_i$ sets $\Box_i \in \mathbb{B}_i$ such that $Y_{\Box_i}$ intersects $R^{2\delta}Q$. 

Let $\{B\}_B$ be a finitely overlapping cover of $B(0,R)$ by balls $B$ of radius $R^{3/4}$ such that each cube $Q \subseteq Y$ is completely contained in some $B$. Define
\[ \mathbb{B}_{i,B} = \left\{ \Box_i \in \mathbb{B}_i :  \Box_i \cap R^{2\delta} B \neq \emptyset \right\}. \] Invoking the definition of each $\mu_i$ in the left hand side of \eqref{sufficient} yields
\begin{align*}  N \prod_{i=1}^k  \mu_i^{\frac{1}{k-1}} &\lessapprox \frac{1}{R^{\frac{d+1}{2}}} \sum_{Q \subseteq Y'}  \int_Q \prod_{i=1}^k \mu_i^{\frac{1}{k-1}} \, dx\\
&\lesssim  \frac{1}{R^{\frac{d+1}{2}}} \sum_B \sum_{Q \subseteq B \cap Y'} \int_Q  \prod_{i=1}^k\left( \sum_{\Box_i \in \mathbb{B}_{i,B}} \chi_{Z_{\Box_i}} \right)^{\frac{1}{k-1}} \, dx \\
&\leq  \frac{1}{R^{\frac{d+1}{2}}}\sum_B \int_B  \prod_{i=1}^k\left( \sum_{\Box_i \in \mathbb{B}_{i,B}} \chi_{Z_{\Box_i}} \right)^{\frac{1}{k-1}} \, dx.\end{align*}

 Each $\chi_{Z_{\Box_i}}$ in the product can be bounded by $\sum_{T \in \mathbb{T}_i} \chi_T$ where each $T$ is a tube of radius $R^{1/2+5\delta}$ and infinite length in the long direction of $\Box_i$ through the centre line of some $S \subseteq Y_{\Box_i}$ which intersects $R^{3\delta}B$. For each $B$, all such sets $S$ come from $\lessapprox 1$ slab of width $\approx R^{3/4}$ in $\Box_i$, so there are $\lessapprox \sigma_{\Box,i}$ such tubes $T \in \mathbb{T}_i$, for each $i$. Therefore, the $k$-linear Kakeya inequality \eqref{kakeya} with $n=d+1$ and exponent close to $\frac{k}{k-1}$ implies
\begin{align} \label{halfway}  N \prod_{i=1}^k  \mu_i^{\frac{1}{k-1}} &\lessapprox \left(\prod_{i=1}^k \sigma_{\Box,i}^{\frac{1}{k-1}} \right)  \cdot \sum_B  \left(\prod_{i=1}^k \left\lvert  \mathbb{B}_{i,B} \right\rvert^{\frac{1}{k-1}}\right). \end{align}

For each $\Box_i \in \mathbb{B}_{i}$, let $T_{\Box_i}$ be a tube of infinite length and radius $R^{3/4+5\delta}$ in the long direction of $\Box_i$, with the same centre line as $\Box_i$. The sum on the right hand side of \eqref{halfway} satisfies
\begin{align*} \sum_B \left(\prod_{i=1}^k\left\lvert  \mathbb{B}_{i,B} \right\rvert^{\frac{1}{k-1}}\right) &\leq \frac{1}{R^{\frac{3(d+1)}{4}}}  \sum_B \int_B \prod_{i=1}^k \left( \sum_{\Box_i \in \mathbb{B}_{i,B}} \chi_{T_{\Box_i}} \right)^{\frac{1}{k-1}} \, dx \\
&\lesssim \frac{1}{R^{\frac{3(d+1)}{4}}}   \int_{B(0,2R)} \prod_{i=1}^k \left( \sum_{\Box_i \in \mathbb{B}_{i}} \chi_{T_{\Box_i}} \right)^{\frac{1}{k-1}} \, dx \\
&\lessapprox \prod_{i=1}^k \left\lvert  \mathbb{B}_i \right\rvert^{\frac{1}{k-1}}, \end{align*}
again by the $k$-linear Kakeya inequality. Combining this with \eqref{halfway} and \eqref{aim} gives 
\[ \left\lVert  \prod_{i=1}^k \left\lvert Ef_i \right\rvert^{1/k} \right\rVert _{L^{q}(Y)} \lessapprox  N^{-\frac{(k-1)\gamma}{k}} \prod_{i=1}^{k} \lVert f_i\rVert _2^{1/k}, \]
which finishes the proof.  \end{proof}

\section{Decay of conical averages} \label{conicalaverages}
This section gives a lower bound on the decay rate $\beta\big(\alpha,\Gamma^d\big)$ of the conical averages, followed by a proof of the relationship between $\beta\big(\alpha,\Gamma^d\big)$ and $s_d(\alpha, 2)$. The section concludes with an upper bound on $\beta\big(\alpha,\Gamma^d\big)$.

Analogously to Corollary~3.3 of \cite{du}, the following lower bound on $\beta\big(\alpha,\Gamma^d\big)$ follows from Theorem~\ref{strichartzr}; the proof is included for completeness.
\begin{corollary} \label{averages} For any $\alpha \in (0,d+1]$,
\[ \beta\big(\alpha,\Gamma^d\big) \geq \alpha -1 + \frac{d-\alpha}{d+1}. \]
\end{corollary}
\begin{proof} Without loss of generality, it may be assumed that $\Gamma$ is restricted to $\xi_d \geq 0$. Recall that $\pi: \mathbb{R}^{d+1} \to \mathbb{R}^d$ removes the last coordinate. If $\mu$ is a Borel measure supported in the unit ball with $c_{\alpha}(\mu) < \infty$, then for any $R>0$, 
\begin{equation} \label{duality} \left(\int \left\lvert \widehat{\mu}(R\xi) \right\rvert^2 \, d\sigma_{\Gamma}(\xi) \right)^{1/2} = \sup_{\lVert f\rVert _2 = 1} \left\lvert  \int \widehat{\mu}(R \xi) \left(f \circ \pi\right)(\xi) \, d\sigma_{\Gamma}(\xi) \right\rvert. \end{equation}
For any fixed $f$ supported in $\{\xi_d \geq 0\} \cap B(0,2) \setminus B(0,1)$ with $\lVert f\rVert _2=1$, 
\begin{align} \notag \left\lvert  \int \widehat{\mu}(R \xi) \left(f \circ \pi\right)(\xi) \, d\sigma_{\Gamma}(\xi) \right\rvert&= \left\lvert  \int \widehat{\mu}(\xi) \left(f \circ \pi\right)(\xi/R) \, dR_{\#}\sigma_{\Gamma}(\xi) \right\rvert \\
\notag &= \left\lvert  \int Ef(-Rx,-Rt) \, d\mu(x,t) \right\rvert \\
\label{duality2} &\leq R^{-\alpha} \lVert Ef\rVert _{L^1(\mu_R)},  \end{align}
where $\mu_R(E) := R^{\alpha}\mu(-R^{-1}E)$ for any Borel set $E$. The measure $\mu_R$ satisfies $c_{\alpha}(\mu_R) \leq c_{\alpha}(\mu)$. If $\psi$ is a fixed non-negative Schwartz function equal to 1 on the ball $B\left( 0, 3\sqrt{2}\right) \supseteq \Gamma$, and vanishing outside a slightly larger ball, then $\widehat{Ef} = \widehat{Ef}  \psi$, which implies that $\lvert Ef\rvert \leq \lvert Ef\rvert \ast \left\lvert  \widecheck{\psi} \right\rvert$. Hence by Fubini and the Schwartz decay of $\widecheck{\psi}$, 
\[ \lVert Ef\rVert _{L^1(\mu_R)} \lesssim_N \int_{B(0,2R)} \left\lvert  Ef\right\rvert H \, dx + R^{-N} \lVert f\rVert _2, \quad \text{where} \quad H := \left\lvert \widecheck{\psi}\right\rvert \ast \mu_R. \]
By the uncertainty principle, the function $H$ satisfies 
\[ \lVert H\rVert _{\infty} \lesssim c_{\alpha}(\mu), \quad \lVert H\rVert _1 \lesssim \lVert \mu\rVert R^{\alpha} , \quad \text{and} \quad \int_{B((x,t),r)} H \, dy \lesssim c_{\alpha}(\mu) r^{\alpha}, \]
for every $(x,t) \in \mathbb{R}^{d+1}$ and $r>0$ (see e.g.~Lemma~4.1 in \cite{erdogan3}; for $r< 1$ the third inequality follows from the first one). Cover $B(0,2R)$ with $\sqrt{2R} A^* \mathbb{Z}^{d+1}$-lattice cubes $Q$, and sort each cube $Q$ into sets $\mathbb{S}_{\lambda}$ according to the dyadic value $\lambda$ of $\lVert Ef\rVert _{L^q(Q)}$. For each $\lambda$, sort the cubes $Q$ in $\mathbb{S}_{\lambda}$ further according to the number of cubes $\sim \sigma$ in $\mathbb{S}_{\lambda}$ lying in the $(2R)^{1/2}$-slab parallel to $A^*\left( \mathbb{R}^d \times \{0\} \right)$ which contains $Q$. Since there are $\lesssim \left( \log R \right)^2$ significant dyadic pairs $(\lambda, \sigma)$, by standard pigeonholing and the triangle inequality there exists a pair $(\lambda, \sigma)$ with associated cubes $Y = \bigcup Q$ such that 
\[ \left\lVert  Ef H^{1/2} \right\rVert _{L^2(B(0,2R))} \lessapprox_N \left\lVert  Ef H^{1/2} \right\rVert _{L^2(Y)}+ c_{\alpha}(\mu)^{1/2} R^{-N} \|f\|_2. \]
Therefore, with $\gamma= \frac{1}{2}- \frac{1}{q}$, applying Theorem~\ref{strichartzr} gives 
\begin{align*} &\int_{B(0,2R)} \left\lvert  Ef\right\rvert H \, dx \\
&\quad\lesssim \lVert \mu\rVert ^{1/2} R^{\alpha/2} \left\lVert  Ef H^{1/2} \right\rVert _{L^2(B(0,2R))} \\
&\quad\lessapprox_N \lVert \mu\rVert ^{1/2}  R^{\alpha/2} \left\lVert  Ef H^{1/2} \right\rVert _{L^2(Y)} + \|\mu\|^{1/2} c_{\alpha}(\mu)^{1/2} R^{-N} \|f\|_2 \\
&\quad\leq \lVert \mu\rVert ^{1/2}  R^{\alpha/2}  \left\lVert  Ef \right\rVert _{L^q(Y)} \left\lVert  H^{1/2} \right\rVert _{L^{1/\gamma}(Y)} + \|\mu\|^{1/2} c_{\alpha}(\mu)^{1/2} R^{-N} \|f\|_2  \\
&\quad\lessapprox \lVert \mu\rVert ^{1/2}c_{\alpha}(\mu)^{1/2}  R^{\alpha/2}  \sigma^{-\gamma} \lVert f\rVert _2 \left( R^{1/2} \sigma  R^{\alpha/2} \right)^{\gamma}\\ 
&\quad=  \lVert \mu\rVert ^{1/2}c_{\alpha}(\mu)^{1/2} R^{ \frac{1}{2} \left( \alpha + \gamma + \alpha \gamma\right) }\lVert f\rVert _2. \end{align*}
Substituting this into \eqref{duality} and \eqref{duality2} yields
\[ \int \left\lvert \widehat{\mu}(R\xi) \right\rvert^2 \, d\sigma_{\Gamma}(\xi) \lessapprox \lVert \mu\rVert c_{\alpha}(\mu) R^{- \alpha+ \gamma + \alpha \gamma } = \lVert \mu\rVert c_{\alpha}(\mu) R^{ -\left(\alpha-1 + \frac{d-\alpha}{d+1}\right)}, \] 
which proves the lower bound. \end{proof}

\begin{corollary} \label{fractalstrichartz}For any $\alpha \in (0,d+1]$,
 \[ s_d(\alpha,2) \leq \frac{d}{2} - \frac{1}{2} \left(\alpha-1+ \frac{d-\alpha}{d+1} \right). \]
\end{corollary}
Corollary~\ref{fractalstrichartz} is a consequence of Corollary~\ref{averages} combined with the following relation between $s_d(\alpha,2)$ and $\beta\big(\alpha,\Gamma^d\big)$ from \cite{wolff}, the proof is included here for convenience, but is essentially the same as the one in \cite{rogers}.

\begin{proposition} \label{relationship} For any $\alpha \in (0,d+1]$, 
\[ \beta\big(\alpha,\Gamma^d\big) = d-2s_d(\alpha,2). \]

\end{proposition}
\begin{proof} Assume now that $E$ denotes the extension operator for the unrestricted cone, and let $\mu$ be a finite Borel measure with support in the unit ball. For any $R \geq 1$ and fixed $f$ with $\supp f \subseteq B(0,2) \setminus B(0,1)$ and $\lVert f\rVert _2=1$, 
\begin{align*} \left\lvert  \int \widehat{\mu}(R \xi) \left(f \circ \pi\right)(\xi) \, d\sigma_{\Gamma}(\xi) \right\rvert&= \left\lvert  \int \widehat{\mu}(\xi) \left(f \circ \pi\right)(\xi/R) \, dR_{\#}\sigma_{\Gamma}(\xi) \right\rvert \\
&= R^{-d}\left\lvert  \int Ef_R \, d\mu \right\rvert \\
&\leq R^{-d} \left\lVert Ef_R \right\rVert _{L^1(\mu)},  \end{align*}
where $f_R(\xi) := f(-\xi/R)$. Since $\supp f_R \subseteq B(0,2R) \setminus B(0,R)$, applying the definition of $s_d(\alpha,1)$ to the wave equation with solution $Ef_R$ gives
\[  R^{-d}\left\lVert Ef_R \right\rVert _{L^1(\mu)} \lessapprox R^{- \left( \frac{d}{2} -s_d(\alpha,1) \right) } \left\lVert \mu\right\rVert ^{1/2}c_{\alpha}(\mu)^{1/2}. \]
Taking the supremum over $\lVert f\rVert _2 =1$ with $\supp f \subseteq B(0,2) \setminus B(0,1)$ yields
\[ \int \left\lvert \widehat{\mu}(R \xi)\right\rvert^2 \, d\sigma_{\Gamma}(\xi) \lessapprox R^{- \left(d-2s_d(\alpha,1) \right) } \left\lVert \mu\right\rVert  c_{\alpha}(\mu). \]
This proves the inequality $\beta\big(\alpha,\Gamma^d\big) \geq d-2s_d(\alpha,1)$. But $s_d(\alpha,1) \leq s_d(\alpha,2)$ by Cauchy-Schwarz, and so
\[ \beta\big(\alpha,\Gamma^d\big) \geq d-2s_d(\alpha,2). \]

To prove this inequality in the reverse direction it suffices to show that $\left\lVert  E \widehat{f} \right\rVert _{L^2\left(\mu\right)} \lesssim_{\epsilon} c_{\alpha}(\mu)^{1/2} \lVert f\rVert _{H^{\frac{d-\beta(\alpha,\Gamma^d)}{2}+\epsilon}}$ for any $\epsilon >0$, since any solution to the wave equation can be expressed through the extension operator. Let $\epsilon >0$ and $f$ be given, and for each fixed dyadic $R \geq 1$ let $g= \widehat{f} \chi_{B(0,2R) \setminus B(0,R) }$. The characterisation of the $L^2(\mu)$ norm through the distribution function is
\begin{equation} \label{stars2} \int \left\lvert  Eg \right\rvert^2 \, d\mu = 2\int_0^{\infty} \lambda \mu\left\{ \lvert Eg\rvert > \lambda \right\} \, d\lambda. \end{equation}
For each $\lambda>0$, define the probability measure $\mu_{\lambda}$ by
\[ \mu_{\lambda}(F) = \frac{ \mu\left( F \cap \left\{ \lvert Eg\rvert > \lambda \right\}\right)}{\mu\left\{ \lvert Eg\rvert > \lambda \right\}}, \]
for any Borel set $F$ (assuming $\lambda$ is not so large that the denominator vanishes). 

For any finite Borel measure $\nu$ supported in the unit ball, 
\begin{multline*} \left\lvert  \int Eg \, d\nu \right\rvert = \left\lvert  \int (\overline{g} \circ \pi)(\xi)  \widehat{\nu}(\xi)  \, d\sigma_{\Gamma}(\xi) \right\rvert \\
\lessapprox  \left(  \left\lVert \nu \right\rVert  c_{\alpha}\left(\nu\right)  R^{d-\beta(\alpha,\Gamma^d)}\right)^{1/2}\lVert g\rVert_2 =   \left\lVert \nu \right\rVert ^{1/2} c_{\alpha}\left(\nu\right)^{1/2} R^{\frac{d-\beta(\alpha,\Gamma^d)}{2}}\lVert g\rVert_2. \end{multline*}
By repeating this with $\nu$ replaced by the restriction of $\nu$ to a set where $\arg Eg$ is essentially constant, the absolute value can be moved inside the integral to get
\[ \left\lVert Eg\right\rVert _{L^1\left(\nu \right)} \lessapprox  \left\lVert \nu \right\rVert^{1/2} c_{\alpha}\left(\nu\right)^{1/2} R^{\frac{d-\beta(\alpha,\Gamma^d)}{2}}  \lVert g\rVert_2. \]
Taking $\nu = \mu_{\lambda}$ gives
\begin{align}\notag \lambda &\leq \left\lVert  Eg\right\rVert _{L^1\left(\mu_{\lambda}\right)} \\
\notag &\lessapprox c_{\alpha}\left(\mu_{\lambda}\right)^{1/2} R^{\frac{d-\beta(\alpha,\Gamma^d)}{2}} \lVert g\rVert _2  \\
\label{weak} &\leq \left( \frac{ c_{\alpha}(\mu)}{ \mu\left\{ \lvert Ef\rvert > \lambda \right\} } \right)^{1/2} R^{\frac{d-\beta(\alpha,\Gamma^d)}{2}} \lVert g\rVert _2. \end{align}
Assume for the moment that $\lVert g\rVert _2 =1$. By \eqref{stars2} and the trivial bound $\lVert Eg\rVert _{\infty} \lesssim R^{d/2}$,
\begin{align} \notag \int \left\lvert  Eg \right\rvert^2 \, d\mu &= 2\int_0^{R^{\frac{d-\beta(\alpha,\Gamma^d)}{2}}} \lambda \mu\left\{ \lvert Eg\rvert > \lambda \right\} \, d\lambda +  2\int_{R^{\frac{d-\beta(\alpha,\Gamma^d)}{2}}}^{CR^{d/2}} \lambda \mu\left\{ \lvert Eg\rvert > \lambda \right\} \, d\lambda \\
\label{stars3} &\lessapprox R^{d-\beta(\alpha,\Gamma^d)} c_{\alpha}(\mu)\lVert g\rVert _2^2, \end{align}
where the bound on the first term follows from $\|\mu\| \leq c_{\alpha}(\mu)$, and the second comes from \eqref{weak}. By scaling this holds even if $\lVert g\rVert _2 \neq 1$. Applying the triangle inequality to \eqref{stars3} gives
\begin{multline*} \left\lVert  E \widehat{f} \right\rVert _{L^2\left(\mu\right)} \lesssim_{\epsilon} c_{\alpha}(\mu)^{1/2} \left\lVert  \widehat{f} \chi_{B(0,1) } \right\rVert _2  \\
+ \sum_{\substack{j=0 \\ R = 2^j}}^{\infty} R^{\frac{d-\beta(\alpha,\Gamma^d)+\epsilon}{2}} c_{\alpha}(\mu)^{1/2} \left\lVert  \widehat{f} \chi_{B(0,2R) \setminus B(0,R) }\right\rVert _2. \end{multline*}
The $\ell^2$ Cauchy-Schwarz inequality then yields
\begin{align*} \left\lVert  E \widehat{f} \right\rVert _{L^2\left(\mu\right)} \lesssim_{\epsilon} c_{\alpha}(\mu)^{1/2} \lVert f\rVert _{H^{\frac{d-\beta(\alpha,\Gamma^d)}{2}+\epsilon}}. \end{align*}
Letting $\epsilon \to 0$ gives
\[ \beta\big(\alpha,\Gamma^d\big) \leq d - 2s_d(\alpha, 2), \]
which proves the proposition. \end{proof}

The following calculation suggests that if the class of measures is restricted, there is no relation between conical averages and fractal Strichartz inequalities; for tensor products $\mu \otimes dm$ the optimal fractal Strichartz inequalities are unknown in general \cite{rogers}. 
\begin{proposition} \label{trivial} If $\nu = \mu \otimes \lambda$, where $\mu$ is a Borel measure with compact support in the unit ball of $\mathbb{R}^d$ and $d\lambda = \chi_{[0,1]} \, dm$, where $m$ is the Lebesgue measure on $\mathbb{R}$, then for $0< \alpha < d$, 
\begin{equation} \label{optimal} \int \left\lvert  \widehat{\nu}(R \xi) \right\rvert^2 \, d\sigma_{\Gamma}(\xi) \lesssim I_{\alpha}(\mu) R^{-(\alpha+2)} \quad \text{for all $R >0$,} \end{equation}
where $I_{\alpha}(\mu)$ is defined by \eqref{energydef}. 
\end{proposition}

\begin{proof} The left hand side of \eqref{optimal} is
\begin{align*} \int_{\Gamma} \lvert \widehat{\nu}(R \xi)\rvert^2 \,d\sigma_{\Gamma} &= \int_{B(0,2) \setminus B(0,1)} \lvert \widehat{\nu}(R \xi, R\rvert\xi\lvert )\rvert^2 \, d\xi \\
&= R^{-d} \int_{B(0,2R) \setminus B(0,R)} \lvert \widehat{\nu}(\xi,\rvert\xi\lvert )\rvert^2 \, d\xi \\
&= R^{-d} \int_{B(0,2R) \setminus B(0,R)} \lvert \widehat{\mu}(\xi)\rvert^2\left\lvert \widehat{ \chi_{[0,1]}}(\lvert \xi\rvert)\right\rvert^2 \, d\xi \\
&\lesssim R^{-d} \int_{B(0,2R) \setminus B(0,R)} \lvert \widehat{\mu}(\xi)\rvert^2  \lvert \xi\rvert^{-2} \, d\xi \\
&\sim R^{-(\alpha+2)} \int_{B(0,2R) \setminus B(0,R)} \lvert \xi\rvert^{\alpha-d} \lvert \widehat{\mu}(\xi)\rvert^2 \, d\xi  \\
&\lesssim R^{-(\alpha+2)} I_{\alpha}(\mu). \end{align*}
This proves the proposition. \end{proof}
 
Returning to the case of a general measure, the following upper bound for the decay of the conical averages is based on the counterexample for the spherical case from \cite{luca}. 
\begin{proposition} For $d \geq 5$ and $\alpha \in (0, d+1)$, 
\[ \beta\big(\alpha,\Gamma^d\big) \leq \alpha-1 +  \frac{2(d+1-\alpha)}{d+1}.  \]
\end{proposition}

\begin{proof} Let $\epsilon, \kappa \in (0,1)$ be constants to be chosen later and let $R>1$. Define
\[ \Lambda =  \left( R^{\kappa-1} \mathbb{Z}^{d+1} + B(0,\epsilon R^{-1}) \right) \cap B(0,1), \]
and let $\mu$ be the Lebesgue measure on $\mathbb{R}^{d+1}$ restricted to $\Lambda$.

Assume that $R$ is large enough to ensure that $R^{ \kappa} > 2$, which makes the balls in $\Lambda$ disjoint. If $0 < r \leq \epsilon R^{-1}$ then 
\[ \frac{ \mu(B(x,r)) }{r^{\alpha} } \lesssim  r^{d+1-\alpha} \lesssim R^{\alpha-(d+1)}. \]

If $\epsilon R^{-1} < r \leq R^{ \kappa-1}$ then 
\[ \frac{ \mu(B(x,r)) }{r^{\alpha} } \lesssim  (\epsilon R^{-1})^{(d+1)-\alpha} \sim R^{\alpha-(d+1)}. \]
If $R^{\kappa-1} < r \leq 1$, then choose a positive integer $N$ so that $r \sim NR^{\kappa -1}$, so that $N \lesssim R^{1- \kappa}$. Then 
\[ \frac{ \mu(B(x,r)) }{r^{\alpha} } \lesssim  \frac{ N^{d+1} (\epsilon R^{-1} )^{d+1} }{N^{\alpha} R^{ (\kappa-1) \alpha}} \lesssim R^{-\kappa (d+1)}. \]
Since $\mu$ is supported in the unit ball, this gives
\[ c_{\alpha}(\mu) \lesssim \max\left( R^{-\kappa (d+1)}, R^{\alpha-(d+1)} \right). \]
Choose $\kappa \in (0,1)$ so that $\kappa (d+1) = (d+1)-\alpha$.  Then 
\[ \lVert \mu\rVert  \lesssim R^{\alpha-(d+1)}, \quad c_{\alpha}(\mu) \lesssim R^{\alpha-(d+1)}. \]
By Plancherel, Cauchy-Schwarz and the definition of $\beta\big(\alpha,\Gamma^d\big)$, every function $f$ on the cone with $\lVert f\rVert _{L^2\left(\sigma_{\Gamma} \right)} \leq 1$ satisfies
\begin{equation} \label{decaydefn} \left\lvert  \int \widehat{ f \sigma_{\Gamma}}(Rx) \, d\mu(x) \right\rvert \lessapprox R^{\alpha-(d+1)} R^{-\beta(\alpha,\Gamma^d)/2}; \end{equation}
a suitably chosen $f$ will give the upper bound on $\beta(\alpha,\Gamma^d)$. 

Let 
\[ E = \left\{ (\xi, \lvert \xi\rvert) \in \Gamma^d : R^{\kappa} (\xi, \lvert \xi\rvert )  \in \mathbb{Z}^{d+1} \right\}. \]
By adding up the lattice points on each of the $\sim R^{\kappa}$ relevant spheres in $\mathbb{R}^d$ of integer radius $\sim R^{\kappa}$, and applying the number theoretic estimate which says that for $d\geq 5$ there are $\sim R^{\kappa(d-2)}$ lattice points on each sphere \cite[Theorem~20.2]{iwaniec}, the cardinality of $E$ satisfies
\begin{equation} \label{numbertheory} \lvert E\rvert \sim R^{\kappa(d-1)}. \end{equation}
Let 
\[ \Omega  = \left\{ \xi \in \Gamma^d : \dist(\xi, E) \leq \epsilon R^{-1} \right\}, \quad f = \frac{ \chi_{\Omega} }{ \lVert \chi_{\Omega}\rVert _{L^2(\sigma_{\Gamma})}}. \] 
Then for large $R$, 
\[ \sigma_{\Gamma}( \Omega ) \sim  \lvert E\rvert R^{-d}. \]
Let $F = \mathbb{Z}^{d+1} \cap B(0, R^{1-\kappa})$, so that
\begin{multline*} \lVert \chi_{\Omega}\rVert _{L^2(\sigma_{\Gamma})} \int \widehat{ f \sigma_{\Gamma}}(Rx) \, d\mu(x) = \int_{B(0,1)} \int_{\Gamma} e^{ -2\pi i \langle \xi , Rx \rangle } \chi_{\Omega}( \xi) \, d\sigma_{\Gamma}(\xi) \, d\mu(x) \\
= \sum_{n \in F}\sum_{\omega \in E}  \int_{B(R^{\kappa-1}n ,\epsilon R^{-1})} \int_{\Gamma \cap B(\omega,\epsilon R^{-1})} e^{ -2\pi i \langle \xi , Rx \rangle }  \, d\sigma_{\Gamma}(\xi) \, d\mu(x). \end{multline*}
For $(\xi,x)$ in the domain of integration, and the corresponding $(\omega,n)$, Cauchy-Schwarz gives
\begin{align*} \left\lvert  \left\langle \xi, Rx \right\rangle - \left\langle \omega, R^{\kappa} n  \right\rangle \right\rvert &\leq \left\lvert  \langle \xi - \omega, Rx \rangle \right\rvert + \left\lvert  \left\langle \omega, R( x- R^{\kappa-1} n )  \right\rangle \right\rvert \\
&\leq 4\epsilon. \end{align*}
Therefore by the defintion of $E$,
\begin{align*} &\lVert \chi_{\Omega}\rVert _{L^2(\sigma_{\Gamma})} \left\lvert \int \widehat{ f \sigma_{\Gamma}}(Rx) \, d\mu(x) \right\rvert\\
&= \sum_{n \in F} \sum_{\omega \in E} \left\lvert \int_{B(R^{ \kappa -1} n,\epsilon R^{-1})} \int_{\Gamma \cap B(\omega,\epsilon R^{-1})} 1  + O(\epsilon) \, d\sigma_{\Gamma}(\xi)\, d\mu(x)\right\rvert \\
&\sim \lvert E\rvert\lvert F\rvert \left(\epsilon R^{-1}\right)^{d+1} \left(\epsilon R^{-1} \right)^d\left(1+ O(\epsilon) \right). \end{align*}
Hence if $\epsilon$ is taken small enough (not depending on $R$), then
\[ \left\lvert \int \widehat{ f \sigma_{\Gamma}}(Rx) \, d\mu(x)\right\rvert  \sim \lvert E\rvert^{1/2} \lvert F\rvert R^{ -\frac{3d}{2} -1}. \]
Combining with \eqref{decaydefn} and \eqref{numbertheory} gives
\[ \beta\big(\alpha,\Gamma^d\big) \leq  \alpha-1 + \frac{2(d+1-\alpha)}{d+1}, \]
by the definition of $\kappa$. \end{proof}


\begin{thebibliography}{}

\bibitem{bennett}
Bennett,~J., Carbery,~A., Tao,~T.: On the multilinear restriction and Kakeya conjectures.
\newblock Acta Math. \textbf{196}, 261--302 (2006)

\bibitem{bourgain}
Bourgain,~J., Demeter,~C.: The proof of the $l^2$ Decoupling Conjecture.
\newblock Ann.~of Math. \textbf{182}, 351--389 (2015)

\bibitem{cho}
Cho,~C.-H., Ham,~S., Lee,~S.: Fractal Strichartz estimate for the wave equation.
\newblock Nonlinear Anal. \textbf{150}, 61--75 (2017)

\bibitem{li}
Du,~X., Guth,~L., Li,~X.: A sharp Schrödinger maximal estimate in $\mathbb{R}^2$.
\newblock Ann.~of Math. \textbf{186}, 607--640 (2017)

\bibitem{du2}
Du,~X., Guth,~L., Li,~X., Zhang,~R.: Pointwise convergence of Schrödinger solutions and multilinear refined Strichartz estimates.
\newblock Forum Math.~Sigma \textbf{6}, 18 pp (2018)

\bibitem{du}
Du,~X., Guth,~L., Ou,~Y., Wang,~H., Wilson,~B., Zhang,~R.: Weighted restriction estimates and application to Falconer distance set problem.
\newblock arXiv:1802.10186v1 (2018)

\bibitem{zhang}
Du,~X., Zhang,~R.: Sharp $L^2$ estimates of the Schrödinger maximal function in higher dimensions.
\newblock Ann.~of Math. \textbf{189}, 837--861 (2019)

\bibitem{erdogan2}
Erdoğan, M.~B.: A note on the Fourier transform of fractal measures.
\newblock Math.~Res.~Lett. \textbf{11}, 299--313 (2004)

\bibitem{erdogan3}
Erdoğan, M.~B.: On Falconer's distance set conjecture.
\newblock Rev.~Mat.~Iberoam. \textbf{22}, 649--662 (2006)

\bibitem{guth3}
Guth,~L.: A short proof of the multilinear Kakeya inequality.
\newblock Math.~Proc.~Cambridge Philos.~Soc. \textbf{158}, 147--153 (2015)

\bibitem{guth}
Guth,~L.: A restriction estimate using polynomial partitioning. 
\newblock J.~Amer.~Math.~Soc. \textbf{29}, 371--413 (2016)

\bibitem{iwaniec}
Iwaniec,~H., Kowalski,~E.: Analytic number theory.
\newblock American Mathematical Society, Providence, RI (2004)

\bibitem{keel}
Keel,~M., Tao,~T.: Endpoint Strichartz estimates.
\newblock Amer.~J.~Math. \textbf{120}, 955--980 (1998)

\bibitem{liu}
Liu,~B.: An $L^2$ identity and pinned distance problem.
\newblock Geom.~Funct.~Anal. \textbf{29}, 283--294 (2019)

\bibitem{luca}
Luca,~R., Rogers,~K.~M.: Average decay of the Fourier transform of measures with applications.
\newblock J.~Eur.~Math.~Soc. (JEMS) \textbf{21}, 465--506 (2019)

\bibitem{mattila}
Mattila,~P.: Spherical averages of Fourier transforms of measures with finite energy; dimension
of intersections and distance sets.
\newblock Mathematika \textbf{34} 207--228 (1987)

\bibitem{oberlin}
Oberlin,~D.~M.: Packing spheres and fractal Strichartz estimates in $\mathbb{R}^d$ for $d \geq 3$.
\newblock Proc.~Amer.~Math.~Soc. \textbf{134}, 3201--3209 (2006)

\bibitem{ou}
Ou,~Y., Wang,~H.: A cone restriction estimate using polynomial partitioning.
\newblock arXiv:1704.05485v1 (2017)

\bibitem{rogers}
Rogers,~K.~M.: Falconer's distance set problem via the wave equation.
\newblock arXiv:1802.01057v1 (2018)

\bibitem{shayya}
Shayya,~B.: Weighted restriction estimates using polynomial partitioning.
\newblock Proc.~Lond.~Math.~Soc. \textbf{115}, 545--598 (2017)

\bibitem{sjolin}
Sjölin,~P.: Estimates of spherical averages of Fourier transforms and dimensions of sets.
\newblock Mathematika \textbf{40}, 322--330 (1993)

\bibitem{strichartz}
Strichartz,~R.~S.: Restrictions of Fourier transforms to quadratic surfaces and decay of solutions of wave equations.
\newblock Duke Math.~J. \textbf{44}, 705--714 (1977)

\bibitem{wolff}
Wolff,~T.: Decay of circular means of Fourier transforms of measures.
\newblock  Int.~Math.~Res.~Not. \textbf{10}, 547--567 (1999)

\bibitem{wolff2}
Wolff,~T.: Local smoothing type estimates on $L^p$ for large $p$. 
\newblock Geom.~Funct.~Anal. \textbf{10}, 1237--1288 (2000)

\end{thebibliography}
\end{document}